\documentclass[preprint,10pt]{elsarticle}

\journal{Applied Mathematics Letters}

\usepackage{amssymb,amsmath,amsthm}

\newcommand{\beq}[1]{ \begin{equation}\label{#1} }
\newcommand{\eeq}{\end{equation}}
\numberwithin{equation}{section}
\DeclareMathOperator*{\esssup}{ess\,sup}

\usepackage{fullpage}
\newtheorem{guess}{Theorem}
\newtheorem{uess}{Lemma}
\newtheorem{corollary}{Corollary}
\newtheorem{definition}{Definition}
\newtheorem{example}{Example}
\newtheorem{remark}{Remark}
\newtheorem{proposition}{Proposition}

\begin{document}  

\begin{frontmatter}

\title{A new stability test for linear neutral  differential equations}


\author[label1]{Leonid Berezansky}
\author[label2]{Elena Braverman}
\address[label1]{Dept. of Math.,
Ben-Gurion University of the Negev,
Beer-Sheva 84105, Israel}
\address[label2]{Dept. of Math. and Stats., University of
Calgary,2500 University Drive N.W., Calgary, AB, Canada T2N 1N4; e-mail
maelena@ucalgary.ca, phone 1-(403)-220-3956, fax 1-(403)--282-5150 (corresponding author)}



\begin{abstract}
We obtain new explicit exponential stability conditions
for the linear scalar neutral equation with two bounded delays
$
\dot{x}(t)-a(t)\dot{x}(g(t))+b(t)x(h(t))=0,
$
where
$
0\leq a(t)\leq A_0<1$, $0<b_0\leq b(t)\leq B$, 
using the Bohl-Perron theorem and a transformation of the neutral equation 
into a differential equation with an infinite number of delays.
The results are applied to
the neutral logistic equation.
\end{abstract}

\begin{keyword} 
neutral  equations \sep  uniform exponential stability \sep Bohl-Perron theorem \sep variable delays \sep
explicit stability conditions \sep logistic neutral differential equation

\noindent
{\bf AMS Subject Classification:} 
34K40, 34K20, 34K06
\end{keyword}

\end{frontmatter}

\section{Introduction}

Many applied problems lead to neutral differential equations as their mathematical models, for example,
a model of a controlled motion of a rigid body, a distributed network (a long line with tunnel diodes), models of infection diseases, a price model in economic dynamics, see, for example, \cite{Krisztin,KolmMysh,Kuang}.
Though neutral delay differential equations describe important applied models,
from mechanics to disease spread in epidemiology, 
compared to other classes of equations, stability theory
for neutral equations with variable coefficients and delays is not sufficiently developed. 
In particular, there are no explicit stability results for general linear equations but only for particular classes of neutral equations, see \cite{Cahlon,Gil,Gop,Gyori,TangZou} and references therein.


%
The aim of the present paper is to 
obtain stability conditions for the equation
\begin{equation}
\label{1} 
\dot{x}(t)-a(t)\dot{x}(g(t))=-b(t)x(h(t))
\end{equation}
which depend on both delays. 
To this end, we transform \eqref{1} 
into a linear delay differential equation with an infinite number of delays. 
This method has not been  applied before to stability problems,
but used to study oscillation in \cite{BB2}.

As an application, we give local asymptotic stability tests
for the logistic neutral 
equation 
\begin{equation}\label{log1}
\dot{x}(t)=r(t) x(t)\left(1-\frac{x(h(t))-\rho \dot{x}(g(t))}{K} \right),
\end{equation}
where $\rho>0$ corresponds to higher resources consumption by a shrinking population. The model
\begin{equation}\label{log2}
\dot{x}(t)=r_0 x(t)\left(1- \frac{x(t-\tau)-\rho \dot{x}(t-\tau)}{K}\right)
\end{equation}
which is an autonomous version of (\ref{log1}), was studied in \cite{GopLog,FK,Yu1}.

\section{Preliminaries}

We consider scalar delay differential equation (\ref{1})
under the following conditions:\\
(a1) $a, b, g, h$ are Lebesgue measurable, $a$ and $b$ are essentially
bounded on $[0,\infty)$ functions;\\
(a2) $ 0\leq a_0\leq a(t)\leq A_0<1, 0<b_0\leq b(t)\leq B_0$  for all $t\geq t_0\geq 0$ and some fixed $t_0\geq 0$;\\ 
(a3)  $mes~ E=0\Longrightarrow mes~ g^{-1}(E)=0$,
where $mes~E$ is  the Lebesgue
measure of the set $E$;\\
(a4)  $0\leq t-g(t)\leq \sigma$, $0 \leq t-h(t) \leq \tau$ for $t \geq t_0$ and some $\sigma>0$, $\tau>0$ and $t_0 \geq 0$.

Along with  (\ref{1}), we consider for each $t_0 \geq 0$ an initial value problem
\begin{equation}
\label{2}
\dot{x}(t)-a(t)\dot{x}(g(t))+b(t)x(h(t))=f(t), ~t\geq t_0,~
x(t)=\varphi(t), ~ t \leq t_0,~\dot{x}(t)=\psi(t),~ t<t_0, 
\end{equation}
where $f:[t_0,\infty)\rightarrow {\mathbb R}$ is a Lebesgue measurable locally essentially bounded  function,
$\varphi:(-\infty,t_0] \rightarrow {\mathbb R}$ and $\psi :(-\infty,t_0)\rightarrow {\mathbb R}$ 
are Borel measurable bounded functions.

Further, we assume that the above conditions hold without mentioning it.

\begin{definition} 
A locally absolutely continuous on $[t_0,\infty)$
function $x: {\mathbb R} \rightarrow {\mathbb R}$ is called {\bf a solution of problem} (\ref{2}) 
if it satisfies the equation in (\ref{2}) for almost all $t\in [t_0,\infty)$ and
the equalities in (\ref{2})
for $t\leq t_0$.
%
For each $s\geq t_0$ the solution $X(t,s)$ of the problem
\begin{equation}
\label{7.6}
\dot{x}(t)-a(t)\dot{x}(g(t))+b(t)x(h(t))=0, ~x(t)=0,~\dot{x}(t)=0,~t<s,~x(s)=1
\end{equation}
is called {\bf the fundamental function} of equation  (\ref{1}). We assume $X(t,s)=0$ for $0\leq t<s$.
We will say that equation (\ref{1}) is {\bf uniformly exponentially stable} 
if there exist 
$M>0$ and $\gamma>0$ such that 
the solution of problem  (\ref{2}) with $f \equiv 0$
has the estimate 
$\displaystyle 
|x(t)|\leq M e^{-\gamma (t-t_0)} \sup_{t \in (-\infty,  t_0]}(|\varphi(t)|+|\psi(t)|)$, $t\geq t_0$,
where $M$ and $\gamma$ do not depend on $t_0 \geq 0$, $\varphi$ and $\psi$.
The fundamental function $X(t,s)$ of equation (\ref{1}) {\bf has an exponential estimate} if it satisfies
$\displaystyle |X(t,s)|\leq M_0 e^{-\gamma_0(t-s)}$ for some $t\geq s\geq t_0$, 
$M_0>0$ and $\gamma_0>0$. 
\end{definition}
 
For a fixed bounded interval $J=[t_0,t_1]$, consider the space $L_{\infty}[t_0,t_1]$ of all essentially bounded on $J$
functions with the 
norm $|y|_J= \esssup_{t\in J} |y(t)|$,
denote $\|f\|_{[t_0,\infty)}=\esssup_{t\geq t_0} |f(t)|$ for an unbounded interval, $I$ is the identity operator.
Define an operator 
on the space $L_{\infty}[t_0,t_1]$ as 
$\displaystyle 
(Sy)(t)=\left\{\begin{array}{ll}
a(t)y(g(t)),& g(t)\geq t_0,\\
0,& g(t)<t_0.\\
\end{array}\right. 
$
Note that there exists a unique
solution of
problem (\ref{2}),
and it 
can be presented as 
\begin{equation*}
x(t)  =  X(t,t_0)x_0+\int_{t_0}^t X(t,s)[(I-S)^{-1}f](s)ds 
+  \int_{t_0}^t X(t,s)[(I-S)^{-1}F](s)ds,
\end{equation*}
where 
$F(t)=a(t)\psi(g(t))-b(t)\varphi(h(t))$ and $\psi(g(t))=0$ for $g(t)\geq t_0$,
$\varphi(h(t))=0$ for $h(t)\geq t_0 $, see, for example, \cite{AzbSim}.

Existence of an exponential estimate for the fundamental function is equivalent \cite{AzbSim} 
to the exponential stability for equations with bounded delays. 
The following result is usually referred to as the Bohl-Perron principle.

\begin{uess}\label{lemma3}\cite[Theorem 4.7.1]{AzbSim}
Assume that 
the solution of the problem 
\begin{equation}\label{10}
\dot{x}(t)-a(t)\dot{x}(g(t))+b(t)x(h(t))=f(t),~~t \geq t_0, ~x(t)=0,~t\leq t_0,~\dot{x}(t)=0,~t<t_0
\end{equation}
is bounded on $[t_0,\infty)$ for any 
essentially bounded on $[t_0,\infty)$ function $f$. 
Then equation (\ref{1}) is uniformly exponentially stable.
\end{uess}

In Lemma~\ref{lemma3} we can consider
boundedness of solutions not for all essentially bounded on $[t_0,\infty)$ functions $f$ but only
for essentially bounded on $[t_0+\delta,\infty)$ functions $f$ that vanish on $[t_0, t_0+\delta)$ for any fixed $\delta>0$, see \cite{BB3}.
We further use this fact in the paper without an additional reference.

Denote by $X_0(t,s)$ the fundamental function of
the equation with a single delay 
\begin{equation}\label{8}
\dot{x}(t)+B(t)x(h_0(t))=0, ~~B(t)\geq 0,~~ 0\leq t-h_0(t)\leq \tau_0.
\end{equation}

\begin{uess}\label{lemma4}\cite{BB3}
If $X_0(t,s)>0$ for $t\geq s\geq t_0$ then 
$\displaystyle 
\int_{t_0+\tau_0}^t X_0(t,s) B(s)\, ds\leq 1.
$
\end{uess}

\begin{uess}\label{lemma5}\cite{BB3,GL}
If for some $t_0\geq 0$,
$\displaystyle 
 \int_{h_0(t)}^t B(s)\, ds\leq \frac{1}{e},~t\geq t_0
$
then $X_0(t,s)>0$ for $t\geq s\geq t_0$.
If in addition $B(t)\geq b_0>0$ then equation (\ref{8}) is 
exponentially stable.
\end{uess}

Finally, the properties of the operator $S$ are outlined in
the following lemma.

\begin{uess}\label{lemma6} \cite{ABR}
If $\|a\|_{[t_0,\infty)}\leq A_0<1$ then $I-S$ is invertible in the space $L_{\infty}[t_0,\infty)$,
 we have $ \displaystyle 
((I-S)^{-1}y)(t)= y(t)+ \sum_{j=1}^{\infty} \prod_{k=0}^{j-1} a\left( g^{[k]}(t) \right) y \left( g^{[j]}(t) \right),
$
where $g^{[0]}(t)=t$, 
$g^{[1]}(t)=g(t)$, $g^{[n]}(t)=g((g^{[n-1]}(t)))$,   and the operator norm satisfies
\begin{equation}
\label{star}
\|(I-S)^{-1}\|_{L_{\infty}[t_0,\infty)\to L_{\infty}[t_0,\infty)}\leq \frac{1}{1-A_0}.
\end{equation}
\end{uess}

\section{Explicit Stability Conditions}


\begin{guess}\label{theorem1}
Assume that for $t\geq t_0$ at least one of the following conditions holds:
\\
a)
$\displaystyle
\tau B_0+\frac{\sigma A_0 B_0^2 (1-a_0)}{(1-A_0)^2 b_0}<1-A_0;
$
\\
b) 
$\displaystyle
t-h(t)\geq \frac{1-A_0}{e B_0}$ ~~ and ~~ $\displaystyle \tau B_0+\frac{\sigma A_0 B_0^2 (1-a_0)}{(1-A_0)^2 b_0}<\left(1+\frac{1}{e}\right)(1-A_0).
$
\\
Then equation (\ref{1}) is uniformly exponentially stable.
\end{guess}
\begin{proof}
Applying $(I-S)^{-1}$ to 
(\ref{10}), using (\ref{star}) on $J$ instead of $[t_0,\infty)$ and (a2), we get
\begin{equation}
\label{2star}
|\dot{x}|_J \leq \left. \left. \left\|(I-S)^{-1} \right\|_{L_{\infty}(J)\to L_{\infty}(J)} \right[   B_0 |x|_J + |f|_J \right] \leq
\frac{B_0}{1-A_0}|x|_J+M_1,
\end{equation} 
where  $\displaystyle M_1=\frac{\|f\|_{[t_0,\infty)}}{1-A_0}$ and
$\|f\|_{[t_0,\infty)}<\infty$. By Lemma \ref{lemma6}, (\ref{10}) 
is equivalent to the equation with an infinite number of delays
\begin{equation}\label{11}
\dot{x}(t)=
- b(t)x(h(t)) -  \sum_{j=1}^{\infty} \prod_{k=0}^{j-1} a\left( g^{[k]}(t) \right) b \left( g^{[j]}(t) \right)
x \left( h(g^{[j]}(t) ) \right)+f_1(t), 
\end{equation}
where $~ f_1(t)=((I-S)^{-1}f)(t)$ and $\|f_1\|_{[t_0,\infty)}<\infty$.
Since $x(t)=0$ for $t\leq t_0$, we can assume that $b(t)=b_0$, $t\leq t_0$.
Denote
$$\displaystyle 
B(t)=b(t) + \sum_{j=1}^{\infty} \prod_{k=1}^{j-1} a\left( g^{[k]}(t) \right) b \left( g^{[j]}(t) \right).
$$
By Lemma~\ref{lemma6}, using the bounds for $a$ and $b$, we obtain
$ \displaystyle \frac{b_0}{1-a_0}\leq B(t)\leq \frac{B_0}{1-A_0}.$
Equation (\ref{11}) can be rewritten in the form 
$$
\dot{x}(t)+B(t)x(t)= b(t)\int_{h(t)}^t \!\!\!\! \dot{x}(\xi)d\xi+
\sum_{j=1}^{\infty} \prod_{k=0}^{j-1} a\left( g^{[k]}(t) \right) b \left( g^{[j]}(t) \right)
\int_{h(g^{[j]}(t))}^t \!\!\!\!\!\! \dot{x}(\xi)d\xi+ f_1(t),
$$
therefore
$$
x(t)=\int\limits_{t_0}^t e^{-\int\limits_s^t B(\xi)d\xi}B(s) \left( \frac{1}{B(s)}\left[b(s) \!\! \int\limits_{h(s)}^s 
\!\! \dot{x}(\xi)d\xi\right.
\left.+\sum_{j=1}^{\infty} \prod_{k=0}^{j-1} a\left( g^{[k]}(s) \right) b \left( g^{[j]}(s) \right)
\!\!\!\!\!\! \int\limits_{h(g^{[j]}(s))}^s 
 \!\!\!\!\!\! \dot{x}(\xi)d\xi\right] \right)ds+ f_2(t),
$$
where $B(t)\geq b_0>0$ implies $\displaystyle \|f_2\|_{[t_0,\infty)} \leq \int_{t_0}^t e^{-\int_s^t B(\xi)d\xi} |f_1(s)|\, ds <\infty$.
We have  
$$
t-g(t)\leq \sigma, t-g^{[2]}(t)=t-g(t)+(g(t)-g(g(t)))\leq 2\sigma,\dots, t-g^{[n]}(t)\leq n\sigma,
$$$$
t-h(t)\leq \tau, t-h(g(t))=t-g(t)+(g(t)-h(g(t)))\leq \sigma+\tau,\dots, t-h(g^{[n]}(t))\leq n\sigma+\tau.
$$
Hence for $t\in J$, 
\begin{align*}
& \frac{1}{B(s)}\left[b(s)\int_{h(s)}^s \dot{x}(\xi)d\xi+
\sum_{j=1}^{\infty} \prod_{k=0}^{j-1} a\left( g^{[k]}(s) \right) b \left( g^{[j]}(s) \right)
\int_{h(g^{[j]}(s))}^s \dot{x}(\xi)d\xi\right]
\\
\leq & \frac{1}{B(s)}\left[ b(s)\tau+
\sum_{j=1}^{\infty} \prod_{k=0}^{j-1} a\left( g^{[k]}(s) \right) b \left( g^{[j]}(s) \right)(\tau+j \sigma)\right] |\dot{x}|_J
\\
\leq & \left[ \tau+\frac{(1-a_0)A_0B_0\sigma}{b_0}\sum_{j=1}^{\infty} jA_0^{j-1} \right] |\dot{x}|_J
\leq \left[ \tau+\frac{\sigma A_0B_0(1-a_0)}{b_0(1-A_0)^2}\right]\frac{B_0}{1-A_0}|x|_J+M_2,
\end{align*}
where the constant $M_2$ does not depend on $J$, and the last inequality is due to (\ref{2star}). 

By Lemma~\ref{lemma4}, the solution of problem (\ref{10}) satisfies 
$\displaystyle
|x|_J \leq \left[\tau+\frac{\sigma A_0B_0(1-a_0)}{b_0(1-A_0)^2}\right]\frac{B_0}{1-A_0}|x|_J+M_3,
$
 where $M_3$ is a constant not dependent on $J$. 

Condition a) of the theorem implies
$\displaystyle
\left[\tau+\frac{\sigma A_0B_0(1-a_0)}{b_0(1-A_0)^2}\right]\frac{B_0}{1-A_0}<1.
$
Hence $|x(t)|\leq M$ for $t \geq t_0$, for some constant $M$ which does not depend on $J$.
By Lemma \ref{lemma3}, equation (\ref{1}) is uniformly exponentially stable.

Next, assume that the conditions in b) hold. Consider the following delay equation
\begin{equation}\label{12}
\dot{x}(t)+B(t)x\left(t-\frac{1-A_0}{B_0 e}\right)=0.
\end{equation}
Since $B(t)\geq b_0$ and $\displaystyle \frac{B(t)(1-A_0)}{B_0 e}\leq \frac{1}{e}$, by Lemma~\ref{lemma5} 
equation (\ref{12}) is exponentially stable, and its fundamental function is positive:  $X_0(t,s)>0$, 
$t\geq s\geq t_0$.
We have
$$
\tau\geq t-h(t)\geq \frac{1-A_0}{B_0 e}, ~~\tau+\sigma\geq t-h(g(t))\geq \frac{1-A_0}{B_0 e},~~ 
\tau+n\sigma\geq t-h(g^{[n]}(t))\geq \frac{1-A_0}{B_0 e}.
$$
Problem (\ref{10}) is equivalent to (\ref{11}) which has a solution
$$
x(t)=\int_{t_0}^t X_0(t,s)  B(s) \left( \frac{1}{B(s)}\left[b(s) \!\!\!\!\! \int\limits_{h(s)}^{s-\frac{1-A_0}{B_0 e}}
\!\!\!\!\! \dot{x}(\xi) \right) d\xi\right.
\left.+\sum_{j=1}^{\infty} \prod_{k=0}^{j-1} a\left( g^{[k]}(s) \right) b \left( g^{[j]}(s) \right)
\!\!\!\!\! \int\limits_{h(g^{[j]}(s))}^{s-\frac{1-A_0}{B_0 e}} \!\!\!\!\!\!\! \dot{x}(\xi)d\xi\right]ds+f_3(s),
$$
where $f_3(t)=\int_{t_0}^t X_0(t,s) f_1(s)ds$, and $\|f_3\|_{[t_0,\infty)}<\infty$, since (\ref{12}) is exponentially stable. 

By the same  calculations as in a) we have
$$
|x|_J \leq \left(\left[\tau+\frac{\sigma A_0B_0(1-a_0)}{b_0(1-A_0)^2}\right]\frac{B_0}{1-A_0}-\frac{1}{e}\right)|x|_J+M_4,
$$
where $M_4$ does not depend on the interval $J$. 

By the second condition in b), we have
 $\displaystyle
\left[\tau+\frac{\sigma A_0B_0(1-a_0)}{b_0(1-A_0)^2}\right]\frac{B_0}{1-A_0}-\frac{1}{e}<1.
$
Hence $|x(t)|\leq M$ for $t \geq t_0$ for some constant $M$ which does not depend on $I$.
By Lemma \ref{lemma3}, equation (\ref{1}) is exponentially stable.
\end{proof}

Consider now two partial cases of equation (\ref{1}), one with constant coefficients
\begin{equation}\label{13}
\dot{x}(t)-a\dot{x}(g(t))=-bx(h(t)),
\end{equation}
where $a,b$ are positive constants, and another with a non-delayed term
\begin{equation}\label{14}
\dot{x}(t)-a(t)\dot{x}(g(t))=-b(t)x(t).
\end{equation}
\begin{corollary}\label{corollary1}
If either a)
$
\displaystyle \tau b+\frac{\sigma a b }{1-a}<1-a;
$ or b) $\displaystyle t-h(t)\geq \frac{1-a}{e b}$ and $
\displaystyle \tau b+\frac{\sigma a b }{1-a}<\left(1+\frac{1}{e}\right)(1-a)
$
then equation (\ref{13}) is uniformly exponentially stable.
\end{corollary}

\begin{corollary}\label{corollary2}
If ~~
$\displaystyle 
\frac{\sigma A_0 B_0^2 (1-a_0)}{(1-A_0)^3 b_0}<1
$
~ then equation (\ref{14}) is uniformly exponentially stable.
\end{corollary}

\section{Examples and Applications}

First, we illustrate the results obtained in the paper with examples. 

\begin{example}
Equation (\ref{13}) with 
$g(t)\equiv t$ and  variable $h(t)$, 
by Corollary \ref{corollary1}, 
is uniformly exponentially stable if $\displaystyle \frac{b\tau}{1-a}<1+\frac{1}{e}\approx 1.37$.
The well-known Myshkis test establishes stability for $\displaystyle  \frac{b\tau}{1-a}<\frac{3}{2}$, under the assumption that the delay function is continuous. 
Corollary~\ref{corollary1} gives a close estimate for a measurable delay. 
\end{example}

\begin{example}
Consider an equation with a variable coefficient and time-dependent $h(t)$
\begin{equation}
\label{ex2eq1}
\dot{x}(t)-0.6\dot{x}(t-0.1)=-r(1+0.1 \sin t)x(h(t)), ~0.9 \leq t-h(t)\leq 1,
\end{equation}
and its particular case with a constant delay
\begin{equation}
\label{ex2eq2}
\dot{x}(t)-0.6\dot{x}(t-0.1)=-r(1+0.1 \sin t)x(t-1).
\end{equation}
We compare Theorem~\ref{theorem1} with applicable results obtained in \cite{TangZou}.
For both (\ref{ex2eq1}) and (\ref{ex2eq2}), we have $\tau=1$, $\sigma=0.1$, $b_0=0.9r$,
$B_0=1.1r$, $a=0.6$. By Part a) of Theorem~\ref{theorem1}, $r<0.307$ implies exponential stability, while Part b) requires $r>0.149$ for (\ref{ex2eq1}), while $r>0.134$ for (\ref{ex2eq2}), and $r<r_0 \approx 0.420347$.

In \cite{TangZou}, a positive integer $N$ is introduced such that in (\ref{ex2eq2}), $a+\frac{3}{2} a^N =0.6+1.5 \cdot 0.6^N \leq 1$; obviously, $N=3$. 
The first asymptotic stability condition for  (\ref{ex2eq2}) from \cite{TangZou}
$$
\limsup_{t \to \infty} \int_{t-(3\tau+(N-1)\sigma)}^t b(s)\, ds < \frac{3}{2}-2a \left( 1-\frac{1}{4} a \right)=0.48
$$
is satisfied for $r<r_1 \approx 0.14118$, while 
the second sufficient inequality from \cite{TangZou}
$$
\limsup_{t \to \infty}  \int_{t-(\tau+(N-1)\sigma)}^t b(s)\, ds < \frac{3-4a^N}{2(1-a^N)} (1-a) \approx 0.544898
$$
holds for $r<r_2 \approx 0.415025$. Note that in this case Theorem~\ref{theorem1} gives a sharper
estimate $\approx 0.420347$ for $r$; in addition, it provides a sufficient exponential stability condition for (\ref{ex2eq1}),
while \cite{TangZou} for (\ref{ex2eq2}) only. To the best of our knowledge, 
other known conditions are also not applicable to (\ref{ex2eq1}). 
\end{example}

Next, let us apply the results of Theorem~\ref{theorem1} to 
logistic neutral equations (\ref{log1}) and (\ref{log2}),
where $\rho>0$, $K>0$, $t-h(t)\leq \tau$, $t-g(t)\leq \tau$, $0<r_0\leq r(t)\leq R_0$, $r_0 \rho \leq R_0 \rho<1$, 
$r,g$ and $h$ are measurable functions.
Equation (\ref{log2}) was studied in \cite{GopLog, FK, Yu1}. 

\begin{proposition}\label{p1}  \cite{Yu1}
If
$\displaystyle 2r_0|\rho|(2-r_0|\rho|)+r\tau<\frac{3}{2}$
then the positive equilibrium $K$ of equation (\ref{log2}) is  locally  asymptotically stable.
\end{proposition}

Note that the inequalities $2r_0|\rho|(2-r_0|\rho|)<\frac{3}{2}$ and $r_0 \rho<1$ imply  $r_0|\rho|<0.5$. 

\begin{guess}\label{thlog}
If either  a) $\displaystyle
\tau R_0 \rho+\frac{\sigma R_0^2 \rho (1-r_0)}{(1-R_0)^2 r_0}<1-R_0,
$ or \\ b) $\displaystyle t-h(t)\geq \frac{1-R_0}{eR_0 \rho}$ ~~ and ~~
$\displaystyle
\tau R_0 \rho+\frac{\sigma R_0^2 \rho (1-r_0)}{(1-R_0)^2 r_0}<\left(1+\frac{1}{e}\right)(1-R_0)
$
\\
then the positive equilibrium $K$ of equation (\ref{log1}) is locally asymptotically stable.

\end{guess}
\begin{proof}
Substituting $x=y-K$ in (\ref{log1}) leads to
 $\displaystyle
\dot{y}(t)=-\frac{r(t)}{K}(y(t)+K)\left[ y(h(t))-\rho \dot{y}(g(t)) \right],
$
its linearization about the zero equilibrium is 
$\displaystyle \dot{z}(t)=-r(t) \left[ z(h(t))-\rho \dot{z}(g(t)) \right]$.
Applying Theorem~\ref{theorem1} with $a_0=r_0$, $A_0=R_0$, $b_0=r_0 \rho$, $B_0=R_0 \rho$, we deduce that  
the linearization
is exponentially stable, and thus $K$ is locally asymptotically stable.
\end{proof}

\begin{remark}
The fact that exponential stability of the linearized equation implies local (and in some cases even global,
see \cite{BB_NA_2009} and references therein) asymptotic stability of a nonlinear scalar equation was applied 
to conclude the proof of Theorem~\ref{thlog}.
\end{remark}

\begin{corollary}\label{corollarylog}
If either $\displaystyle \tau r_0\rho<  (1-r_0)^2$ or $\displaystyle
\frac{1-r_0}{e} <    \tau r_0\rho < \left(1+\frac{1}{e}\right) (1-r_0)^2$
then the positive equilibrium $K$ of equation (\ref{log2}) is locally asymptotically stable.
%
\end{corollary}

Compared to Proposition~\ref{p1}, 
Theorem~\ref{thlog} 
is applicable to non-autonomous equations with different delays. 
Also, for $r_0=0.2$, $\rho=4$ and any $\displaystyle 0 \leq \tau<0.8 \left(1+\frac{1}{e}\right)$,
Theorem~\ref{thlog} establishes local asymptotic stability of (\ref{log2}), while for these $r_0$ and $\rho$, Proposition~\ref{p1}
fails for any $\tau$.

\section*{Acknowledgment}

The second author was partially supported by the NSERC research grant RGPIN-2015-05976.

\end{document}